\documentclass[12pt]{article}
\usepackage{amssymb,amsfonts,amsmath,amsthm,dsfont,hyperref,xcolor}
\usepackage[letterpaper, portrait, margin=26.5mm]{geometry}
\usepackage[lite,alphabetic]{amsrefs}
\usepackage{soul}

\newcommand{\1}{\mathds{1}}
\newcommand{\0}{\mathds{O}}

\newcommand{\R}{\mathbb{R}}

\newcommand{\N}{\mathbb{N}}

\newcommand{\Bo}{\mathrm{B}}
\newcommand{\Bbo}{\overline{\mathrm{B}}}
\newcommand{\8}{\infty}

\newcommand{\spa}{\mathrm{span}}

\newcommand{\vlt}{\mathrm{vlt}}
\newcommand{\Aff}{\mathrm{Aff}}
\newcommand{\Hom}{\mathrm{Hom}}

\newcommand{\supp}{\mathrm{supp~}}
\newcommand{\Co}{\mathcal{C}}

\newcounter{dummy} \numberwithin{dummy}{section}
\newtheorem{theorem}[dummy]{Theorem}

\newtheorem*{formula*}{Formula}
\newtheorem*{problem*}{Problem}
\newtheorem*{theorem*}{Theorem}
\newtheorem{lemma}[dummy]{Lemma}

\newtheorem{proposition}[dummy]{Proposition}
\newtheorem{corollary}[dummy]{Corollary}
\newtheorem{question}[dummy]{Question}
\theoremstyle{remark}
\newtheorem{remark}[dummy]{Remark}
\newtheorem{example}[dummy]{Example}
\usepackage{scalerel}

\title{A universal approximation theorem and its applications to vector lattice theory}

\author{Eugene Bilokopytov\thanks{Department of Mathematical and Statistical Sciences,
  University of Alberta, Edmonton, AB, T6G\,2G1, Canada (\texttt{bilokopy@ualberta.ca}).} and Foivos Xanthos\thanks{Department of Mathematics, Toronto Metropolitan University, 350 Victoria Street, Toronto, ON, M5B2K3, Canada (\texttt{foivos@torontomu.ca}).}}

\begin{document}

\maketitle

\begin{abstract}
A classical result in approximation theory states that for any continuous function \( \varphi: \mathbb{R} \to \mathbb{R} \), the set \( \operatorname{span}\{\varphi \circ g : g \in \operatorname{Aff}(\mathbb{R})\} \) is dense in \( \mathcal{C}(\mathbb{R}) \) if and only if \( \varphi \) is not a polynomial. In this note, we present infinite dimensional variants of this result. These extensions apply to neural network architectures and improve the main density result obtained in \cite{BDG23}. We also discuss applications and related approximation results in vector lattices, improving and complementing results from \cite{AT:17, bhp,BT:24}.


MSC2020 41A65, 46A19, 46A40, 46B42, 46E25, 46E40, 68T07.
\end{abstract}

\section{Introduction}

The term \emph{universal approximation theorem} covers a family of density results in function spaces, where the approximating sets are determined by a given neural network architecture. We refer the reader to \cite{BDG23,pinkus,pink} and references therein. The results presented in this paper apply directly to the neural network architecture on infinite-dimensional spaces introduced in \cite{BDG23}.

Let $\mathfrak{X}$ be a locally convex space and $\mathfrak{X}'$ its topological dual. Let $A \in \mathcal{L}(\mathfrak{X})$ be a continuous linear operator, $b \in \mathfrak{X}$, $\ell \in \mathfrak{X}'$, and let $\sigma: \mathfrak{X} \rightarrow \mathfrak{X}$ be a continuous function, called the \emph{activation function}. A \emph{neuron} is defined by
\[
\mathcal{N}_{\ell,A,b}: \mathfrak{X} \rightarrow \mathbb{R}, \quad \mathcal{N}_{\ell,A,b}(x) := \ell(\sigma(Ax + b)),
\]
and the corresponding set of \emph{neural networks} is the linear span of all such neurons:
\[
\mathfrak{N}(\sigma): = \operatorname{span}\left\{ \mathcal{N}_{\ell,A,b} \mid \ell \in \mathfrak{X}',~ A \in \mathcal{L}(\mathfrak{X}),~ b \in \mathfrak{X} \right\}.
\]

The search for a universal approximation theorem in this framework leads to the following:

\begin{problem*}\label{problem}
Determine necessary and sufficient conditions on $\sigma$ under which $\mathfrak{N}(\sigma)$ is dense in $C(\mathfrak{X},\mathbb{R})$ with respect to the compact-open topology.
\end{problem*}

In the case where $\mathfrak{X} = \mathbb{R}^m$, and the activation function has the form
\[
\sigma_\phi(x_1,\ldots,x_m) = (\phi(x_1),\ldots,\phi(x_m))
\]
for some continuous function $\phi: \mathbb{R} \rightarrow \mathbb{R}$, the set of neural networks admits the representation
\[
\mathfrak{N}(\sigma_\phi) = \operatorname{span} \left\{ \phi \circ g \mid g \in \Aff(\mathbb{R}^m) \right\},
\]
where $\Aff(\mathbb{R}^m)$ denotes the set of affine functions from $\mathbb{R}^m$ to $\mathbb{R}$. This formulation has been extensively studied, particularly in a series of foundational papers from the 1990s (see Section 3 of \cite{pinkus} for a detailed historical overview). We recall that a function $\phi: \mathbb{R} \rightarrow \mathbb{R}$ is called a \emph{sigmoid function} if $\lim_{t \rightarrow \infty} \phi(t) = 1 \quad \text{and} \quad \lim_{t \rightarrow -\infty} \phi(t) = 0.$ A first major step toward a complete solution to the approximation problem in the finite-dimensional setting is given by the following important result.

\begin{theorem}[Sigmoidal-UAT] \label{cybenko}
Let $\phi: \mathbb{R} \rightarrow \mathbb{R}$ be a continuous sigmoid function. Then $\mathfrak{N}(\sigma_\phi)$ is dense in $C(\mathbb{R}^m)$.
\end{theorem}

In \cite{BDG23}, the authors established an infinite-dimensional analogue of this result (see \cite[Theorem 2.3 and Theorem 2.8]{BDG23}), under the assumption that $\mathfrak{X}$ is a Fréchet space and $\sigma$ is von Neumann bounded and satisfies the \emph{separating property}, which serves as an infinite-dimensional counterpart to the sigmoidal condition (see \cite[Definition 2.6]{BDG23}).

Although the sigmoidal condition is quite general, commonly used activation functions such as the ReLU function, $\operatorname{ReLU}(x) := \max\{x, 0\}, \quad x \in \mathbb{R},$
do not satisfy it. Nevertheless, it is known that $\mathfrak{N}(\sigma_{\operatorname{ReLU}})$ is dense in $C(\mathbb{R}^m)$. The definitive resolution of the Problem in the finite-dimensional case is captured by the following seminal result:

\begin{theorem}[Non-polynomial-UAT] \label{UAT}
Let $\phi: \mathbb{R} \rightarrow \mathbb{R}$ be a continuous function. Then $\mathfrak{N}(\sigma_\phi)$ is dense in $C(\mathbb{R}^m)$ if and only if $\phi$ is non-polynomial.
\end{theorem}

The key step in the proof of this result is the case $m = 1$. Although the theorem is well-known in the neural networks literature, it is less familiar in the functional analysis community. For this reason, we provide in Subsection \ref{subsca} a concise and elementary proof. In Subsection \ref{subinf}, we extend this result to infinite-dimensional settings. In particular, Theorem~\ref{nn} presents an infinite-dimensional analogue of Theorem~\ref{UAT}, thereby resolving the Problem. Our approach draws on techniques from the theory of polynomials on topological vector spaces.\medskip

Section~3 explores applications to vector lattice theory. We emphasize that these applications are specifically related to the ReLU activation function. We refer the reader to
\cite{AL:06,lz,zaanen} for background on the topic. Let $F$ be an Archimedean
vector lattice and let $A \subset F$. We denote by $\vlt(A)$ the linear
sublattice of $F$ generated by $A$, that is, the smallest linear subspace of
$F$ that contains $A$ and is closed under the lattice operation
$x \mapsto x^+$.

As a concrete example, consider the vector lattice $F=C(\mathbb{R})$ of
continuous functions on $\mathbb{R}$ and let $A=\{\1,\mathrm{Id}\}$, where $\1$ is
the constant one function and $\mathrm{Id}(t)=t$ is the identity function. It is known
that $\vlt(A)$ coincides with the space of all piecewise affine functions.
Moreover, these functions can be written in the form $\mathfrak{N}(\sigma_\phi)$
with $\phi$ equal to the ReLU activation function (see, e.g.,
\cite[Chapter 7]{AT:07}).

Since the ReLU function can be defined on an arbitrary vector lattice $F$ via
the lattice operation $x^+$, it is natural to ask whether such a representation
exists for a general space $F$. We also note that this problem is closely
related to the option spanning problem in finance (see, e.g., \cite{gl,GX:17}).
In \cite[Theorem 2.1]{bhp}, the authors provided the following representation
of $\vlt(A)$ in terms of a neural network architecture:

\begin{equation}\tag{$\star$}
\vlt(A) = \operatorname{span}(\operatorname{span}A)^+, \text{ for any }  A = \{u, v\} \subset F_+
\end{equation}


Somewhat surprisingly, Huijsmans \cite{huijsmans} showed that formula~($\star$) fails in general when $A$ contains more than two positive elements, or when $A \not\subset F_+$. In Subsection \ref{subhui}, we leverage the results from Subsection \ref{subinf} to show that ($\star$)  remains valid for arbitrary $A \subset F_+$ if an appropriate closure is taken. Specifically, in Corollary~\ref{formula}, we prove that if $F$ is a Banach lattice, then
\[
\overline{\vlt(A)} = \overline{\operatorname{span}(\operatorname{span} A)^+}, \text{ for any } A \subset F_+,
\]
where the closures are taken in the norm topology.

In Subsections \ref{suboth} and \ref{submisc}, we investigate approximation results in vector lattices. In \cite{AT:17}, the authors studied the density of the sublattice $\vlt(\Aff(\mathbb{R}^m))$ in $C(\mathbb{R}^m)$. In Corollary~\ref{vladimir}, we strengthen \cite[Proposition 4]{AT:17}, and in Theorem~\ref{vladimir2}, we extend the result to infinite-dimensional spaces. Finally, in Corollary~\ref{prop1}, we present a version of the Stone–Weierstrass Theorem for spaces of the form $C(X, E)$, where $X$ is hemi-compact and $E$ is an AM-space. To derive this result we introduced AM-topologies on vector lattices, which are of independent interest. We also complement some results in \cite{BT:24} (see Remark \ref{timur}).

\section{A Universal approximation theorem and its extensions}

\subsection{The scalar case}\label{subsca}

For any Tychonoff space $X$ we always view $\Co\left(X\right)$ to be endowed with the compact-open topology, i.e. the topology of uniform convergence on compact sets. If $X$ is compact this topology is generated by the supremum norm $\|\cdot\|_{\8}$. We denote the constant zero function by $\0$ and the constant one function by $\1$.

Throughout the text, vector spaces are assumed to be non-trivial. Recall that a Hausdorff topological vector space $E$ is a Tychonoff topological space (this follows from e.g. \cite[Theorem 2.7.6]{bn}). We denote its topological dual by $E^{*}$.

If $A\subset\R$ is either open or closed, we denote the closed subspace of $\Co\left(A\right)$ which consists of all polynomials of degree at most $m$ by $P_{m}\left(A\right)$. In particular, $\Aff\left(A\right):=P_{1}\left(A\right)$.

The following result originates in \cite{schwartz}, see also \cite[Example 2.4.4]{edwards}, \cite[Section 3]{pinkus}, \cite[Proposition 5.8]{pink} and \cite{ismailov}. For the reader's convenience we provide a proof, which is somewhat different from all those sources.

\begin{theorem}\label{lybenko}
Let $U\subset\R$ be an open interval, let $K\subset \R$ be compact, and let $\varphi\in\Co\left(U\right)$. Then, $\spa\left\{\varphi\circ g,~ g\in \Aff\left(K\right),~ g\left(K\right)\subset U\right\}$ is dense in $\Co\left(K\right)$ if and only if $\varphi$ is not a polynomial.
\end{theorem}

We start with an auxiliary result.

\begin{lemma}\label{cybe}
Let $U\subset\R$ be an open interval, let $K\subset \R$ be compact, and let $\varphi:U\to\R$ be infinitely differentiable. Assume that $\nu$ is a regular signed Borel measure on $K$ such that $\int\limits_{K} t^{m}d\nu\left(t\right)\ne 0$, but $\int\limits_{K} \varphi\circ gd\nu=0$, for every $g\in \Aff\left(K\right)$ such that $g\left(K\right)\subset U$. Then, $\varphi\in P_{m-1}\left(U\right)$.
\end{lemma}
\begin{proof}
For a fixed $b\in U$ there is $r_{b}>0$ such that $at+b\in U$, whenever $t\in K$ and $\left|a\right|<r_{b}$. Let $\psi_{b}:\left(-r_{b},r_{b}\right)\to\R$ be defined by $\psi_{b}\left(a\right):=\int\limits_{K}\varphi\left(at+b\right)d\nu\left(t\right)$. Since $\varphi$ is infinitely smooth, we have that $\psi_{b}'\left(a\right)=\int\limits_{K} t\varphi'\left(at+b\right)d\nu\left(t\right)$. Iterating the argument, one can show that $\psi_{b}^{\left(n\right)}\left(a\right)=\int\limits_{K} t^{n}\varphi^{\left(n\right)}\left(at+b\right)d\nu\left(t\right)$, for every $n\in\N_{0}$.

Now, recall that $\psi_{b}\equiv 0$, therefore $\psi_{b}^{\left(n\right)}\equiv 0$, and in particular $0=\psi_{b}^{\left(n\right)}\left(0\right)$\linebreak $=\int\limits_{K} t^{n}\varphi^{\left(n\right)}\left(b\right)d\nu\left(t\right)=\varphi^{\left(n\right)}\left(b\right)\int\limits_{K} t^{n}d\nu\left(t\right)$, for every $n\in\N_{0}$ and $b\in U$. Since $\int\limits_{K} t^{m}d\nu\left(t\right)\ne 0$, we conclude that $\varphi^{\left(m\right)}\left(b\right)=0$, for every $b\in U$, which completes the argument.
\end{proof}

\begin{proof}[Proof of Theorem \ref{lybenko}]
Let $G_{\varphi}:=\spa\left\{\varphi\circ g,~ g\in \Aff\left(K\right),~ g\left(K\right)\subset U\right\}$. If $\varphi$ is a polynomial, then $G_{\varphi}\subset P_{\deg \varphi}\left(K\right)$, hence it is not dense. We now prove sufficiency.\medskip

If $G_{\varphi}$ is not dense, there is a nonzero $\nu\in\Co\left(K\right)^{*}$ which vanishes on it. Note that $\nu$ is a regular signed Borel measure (see \cite[Theorem 4.10.1]{edwards}), so that $\int\limits_{K}\varphi\left(at+b\right)d\nu\left(t\right)=0$, for any $a,b\in\R$ such that $aK+b\subset U$. Since polynomials are dense in $\Co\left(K\right)$, there is $m\in\N_{0}$ such that $\int\limits_{K} t^{m}d\nu\left(t\right)\ne 0$. Our goal is to show that $\varphi\in P_{m-1}\left(K\right)$.\medskip

Find an open interval $V\subset U$ and an open neighborhood $W\subset\R$ of $0$ such that $V+\overline{W}\subset U$. Assume that $f:\R\to\R$ is infinitely smooth and that it vanishes outside of $W$. Then, $\varphi_{f}:=\varphi\ast f$ is infinitely smooth on $\R$ (see \cite[Proposition 8.10]{folland}; here we extend $\varphi$ by $0$ outside of $U$). Let $\lambda$ be the Lebesgue measure on $\R$. For every $a,b\in\R$ such that $aK+b\subset V$ we have that
\begin{align*}
\int\limits_{K}\varphi_{f}\left(at+b\right)d\nu\left(t\right)&=\int\limits_{K}\int\limits_{\overline{W}}\varphi\left(at+b-s\right)f\left(s\right)d\lambda\left(s\right)d\nu\left(t\right)\\
=\int\limits_{\overline{W}}\int\limits_{K}\varphi\left(at+b-s\right)f\left(s\right)d\nu\left(t\right)d\lambda\left(s\right)&=\int\limits_{\overline{W}} f\left(s\right)\int\limits_{K}\varphi\left(at+b-s\right) d\nu\left(t\right)d\lambda\left(s\right)\\
&=\int\limits_{\overline{W}} f\left(s\right)\cdot 0 d\lambda\left(s\right)=0,
\end{align*}
where Fubini theorem is applicable since both $f,\varphi$ are continuous, and $K,\overline{W}$ are compact. Lemma \ref{cybe} allows us to conclude that $\left.\varphi_{f}\right|_{V}\in P_{m-1}\left(V\right)$.

For $n\in\N$ let $W_{n}:=\left(-\frac{1}{n},\frac{1}{n}\right)$ and let $f_{n}:\R\to\R$ be infinitely smooth, vanishing outside of $W_{n}$ and such that $\int f_{n}d\lambda=1$. It is not hard to show that in this case $\varphi_{f_{n}}\to \varphi$.

Fix $x\in U$ and an interval $V$ such that $x\in V\subset\overline{V}\subset U$, so that for large enough $n\in\N$ we have $V+\overline{W_{n}}=U$. By the argument above $\left.\varphi_{f_{n}}\right|_{V}\in P_{m-1}\left(V\right)$, for large enough $n$. It follows that $\left.\varphi\right|_{V}\in P_{m-1}\left(V\right)$. We conclude that $\varphi$ is infinitely smooth on a neighborhood of $x$ and $\varphi^{\left(m\right)}\left(x\right)=0$. Since $x$ was arbitrary, the claim follows.
\end{proof}

The following ``global version'' of Theorem \ref{lybenko} is an immediate corollary.

\begin{theorem}\label{cybenko}
Let $\varphi:\R\to\R$ be continuous. Then, $\spa\left\{\varphi\circ g,~ g\in \Aff\left(\R\right)\right\}$ is dense in $\Co\left(\R\right)$ if and only if $\varphi$ is not a polynomial.
\end{theorem}

\subsection{Infinite dimensional variants}\label{subinf}

The following result extends \cite[Theorem 3.1]{pinkus} to the infinite dimensional case (see also \cite[Theorem 1]{ismailov} for a related extension).

\begin{theorem}\label{eugene0}
Let $E$ be a Hausdorff topological vector space, let $F\subset E^{*}$ be a subspace and let $\varphi:\R\to\R$ be continuous. Then $\spa\left\{\varphi\circ \left(f+r\1\right),~ f\in F,~r\in\R\right\}$ is dense in $\Co\left(E\right)$ if and only if $F$ separates points of $E$ and $\varphi$ is not a polynomial.
\end{theorem}
\begin{proof}
Let $G:=\spa\left\{\varphi\circ g,~ g\in \Aff\left(\R\right)\right\}$, and let $H:=\spa\left\{\varphi\circ \left(f+r\1\right),~ f\in F,~r\in\R\right\}$.\medskip

Necessity: First, recall that $E$ is Tychonoff, and so $\Co\left(E\right)$ separates points of $E$. If $F$ does not separate certain two points of $E$, then $H$ does not separate them either, and so the latter cannot be dense in $\Co\left(E\right)$.

Take $e\ne 0_{E}$ and let $\psi:\R\to E$ be defined by $\psi\left(t\right):=te$. The composition with $\psi$ is a continuous operator $C_{\psi}:\Co\left(E\right)\to\Co\left(\R\right)$. Note that for any $h\in \Co\left(\R\right)$ and $r>0$ there is $\hat{h}\in\Co\left(E\right)$ such that $\hat{h}\circ\psi$ agrees with $h$ on $\left[-r,r\right]$ (this follows from Tietze extension theorem applied to the compact set $\left[-r,r\right]e$ in any Hausdorff compactification of $E$). This argument shows that $C_{\psi}$ has a dense image. Therefore, if $H$ is dense in $\Co\left(E\right)$, then $C_{\psi}H$ is dense in $\Co\left(\R\right)$. However, for every $f\in F$ and $r\in\R$ we have that $\left(f+r\1\right)\circ \psi\in \Aff\left(\R\right)$, and so $C_{\psi}\left(\varphi\circ\left(f+r\1\right)\right)=\varphi\circ\left(\left(f+r\1\right)\circ \psi\right)\in G$. Thus, $C_{\psi}H\subset G$ implies denseness of $G$ in $\Co\left(\R\right)$. By Theorem \ref{cybenko} we conclude that $\varphi$ is not a polynomial.\medskip

Sufficiency: By Theorem \ref{cybenko} we have that $\exp \in \overline{G}$. For every $f\in F$ and $g\in G$, we have that $g\circ f\in H$. Since the composition is a continuous operation (see \cite[Theorem 3.4.2]{engelking}), we conclude that $\exp_{f}:=\exp\circ f\in\overline{H}$. Note that $\exp_{f}\exp_{w}=\exp_{f+w}$, for any $f,w\in F$. Furthermore, if $e,u\in E$ are distinct, there is $f\in F$ such that $f\left(e-u\right)\ne 0$, hence $\frac{\exp_{f}\left(e\right)}{\exp_{f}\left(u\right)}=\exp\left(f\left(e-u\right)\right)\ne 1$, therefore $\exp_{f}\left(e\right)\ne\exp_{f}\left(u\right)$. It follows that $\spa\left\{\exp_{f},~f\in F\right\}\subset \overline{H}$ is a subalgebra of $\Co\left(E\right)$ which contains constants (observe that $\1=\exp_{0_{F}}$) and separates points. Stone-Weierstrass theorem (see, e.g. \cite[Theorem~16.5.7]{bn}) yields that $\spa\left\{\exp_{f},~f\in F\right\}$ is dense in $\Co\left(E\right)$, thus $\overline{H}=\Co\left(E\right)$.
\end{proof}

If $E,F$ are topological vector spaces, an \emph{affine map} from $E$ into $F$ is a map of the form $e\mapsto Te+f$, where $T:E\to F$ is a continuous linear operator, and $f\in F$ is fixed. We denote the collection of such maps by $\Aff\left(E,F\right)$, or just $\Aff\left(E\right)$ if $F=\R$. Since in a locally convex space linear functionals separate points, we get the following corollary.

\begin{corollary}\label{Eugenue}
Let $E$ be a locally convex Hausdorff space and let $\varphi\in\Co\left(\R\right)$. Then, \linebreak $\spa\left\{\varphi\circ g,~ g\in \Aff\left(E\right)\right\}$ is dense in $\Co\left(E\right)$ if and only if $\varphi$ is a non-polynomial.
\end{corollary}

Let $E$ and $F$ be vector spaces, and let $\Phi: E \rightarrow F$. We say that $\Phi$ is a \emph{polynomial} of degree $n\in\N$ if there exist mappings $M_{k}: E^{k} \to F$, $k=1, \dots ,n$ such that $M_{k}$ is linear separately in each coordinate, symmetric (i.e., $M\left(e_{1}, \dots, e_{k}\right) = M\left(e_{\pi\left(1\right)}, \dots, e_{\pi\left(k\right)}\right)$ for every permutation $\pi$ of $\left\{1, \dots, k\right\}$), and satisfies $\Phi\left(e\right) = M_{1}\left(e\right)+ \dots + M_{n}\left(e, \dots ,e\right)$ for all $e \in E$. For background material on polynomials between topological vector spaces, we refer the reader to \cite{bs}; see also \cite[Section~1]{hj} for the case of Banach spaces. We note that if $E$ is a non-trivial Banach lattice, then the map $\Phi: E \rightarrow E$ defined by $\Phi(x) = x^+$ is not a polynomial. Moreover, $\Phi \in C(E, E)$, and $\Phi$  does not satisfy the separating property introduced in \cite[Definition~2.6]{BDG23}. The following theorem may be viewed as a significant improvement of the results presented in \cite[Section~2]{BDG23}.

\begin{theorem}\label{nn}
Let $E,F$ be locally convex spaces and let $\Phi\in\Co\left(E,F\right)$. Then, the following conditions are equivalent:
\item[(i)] $\Phi$ is a polynomial;
\item[(ii)] There is $m\in\N$ such that $\nu\circ\Phi\circ\mathbf{h}\in P_{m}\left(\R\right)$, for every $\nu\in F^{*}$ and $\mathbf{h}\in\Aff\left(\R,E\right)$;
\item[(iii)] $\spa\left\{\nu\circ\Phi\circ \mathbf{g},~ \nu\in F^{*},~ \mathbf{g}\in \Aff\left(E,E\right)\right\}$ is not dense in $\Co\left(E\right)$.\medskip

If additionally $E$ and $F^{*}$ (with some topology stronger than weak*) are Baire, then the conditions above are equivalent to the fact that $\nu\circ\Phi\circ\mathbf{h}$ is a polynomial, for every $\nu\in F^{*}$ and $\mathbf{h}\in\Aff\left(\R,E\right)$.
\end{theorem}
\begin{proof}
(i)$\Rightarrow$(ii): If $\Phi$ is a polynomial of degree $m$, then $\nu\circ\Phi\circ\mathbf{h}\in P_{m}\left(\R\right)$, for every $\nu\in F^{*}$ and $\mathbf{h}\in\Aff\left(\R,E\right)$ (see \cite[Fact 38]{hj}).

(ii)$\Rightarrow$(i): Fix $\mathbf{h}\in\Aff\left(\R,E\right)$. As $\nu\circ\Phi\circ\mathbf{h}\in P_{m}\left(\R\right)$, for every $\nu\in F^{*}$, it follows from \cite[Corollary 4]{bs} that $\Phi\circ\mathbf{h}$ is a polynomial of degree at most $m$. As $\mathbf{h}$ was arbitrary, it follows from \cite[Corollary 3]{bs} that $\Phi$ is a polynomial of degree at most $m$.\medskip

(ii)$\Rightarrow$(iii): Assume that the span $V$ in question is dense and let $\psi$ be as in the proof of Necessity in Theorem \ref{eugene0}. The argument from that proof shows that $\left\{v\circ\psi,~ v\in V\right\}$ is dense in $\Co\left(\R\right)$. On the other hand, for every $\nu\in F^{*}$ and $\mathbf{g}\in\Aff\left(E,E\right)$ we have $\mathbf{g}\circ \psi\in \Aff\left(\R,E\right)$, and so by assumption
$\nu\circ\Phi\circ \mathbf{g}\circ \psi\in P_{m}\left(\R\right)$. Hence, $\left\{v\circ\psi,~ v\in V\right\}\subset P_{m}\left(\R\right)$ cannot be dense. Contradiction.\medskip

(iii)$\Rightarrow$(ii): First, assume that $\nu_{0}\in F^{*}$ and $\mathbf{h}_{0}\in\Aff\left(\R,E\right)$ are such that $\varphi:=\nu_{0}\circ\Phi\circ\mathbf{h}_{0}$ is not a polynomial. Then, for every $f\in\Aff\left(E\right)$ we have that $\mathbf{h}_{0}\circ f\in \Aff\left(E,E\right)$, and so
$\left\{\nu\circ\Phi\circ \mathbf{g},~ \nu\in F^{*},~ \mathbf{g}\in \Aff\left(E,E\right)\right\}$ contains $\left\{\varphi\circ f,~  f\in \Aff\left(E\right)\right\}$, whose span is dense by Corollary \ref{Eugenue}, contradiction. Hence, $\nu\circ\Phi\circ\mathbf{h}$ is a polynomial, for every $\nu\in F^{*}$ and $\mathbf{h}\in\Aff\left(\R,E\right)$.

Assume that for every $m\in\N$ there are $\nu_{m}\in F^{*}$ and $\mathbf{h}_{m}\in\Aff\left(\R,E\right)$ such that $\deg \varphi_{m}>m$, where $\varphi_{m}:=\nu_{m}\circ\Phi\circ\mathbf{h}_{m}$. Then, $\overline{\spa}\left\{\varphi_{m}\right\}_{m\in\N}$ in $\Co\left(\R\right)$ must contain a non-polynomial $\varphi$ (here we use the fact that the algebraic dimension of a closed infinite dimensional subspace of the Frechet space $\Co\left(\R\right)$ is uncountable, while it is countable for the space of all polynomials). From the argument above
$\left\{\nu\circ\Phi\circ \mathbf{g},~ \nu\in F^{*},~ \mathbf{g}\in \Aff\left(E,E\right)\right\}$ contains $\left\{\varphi_{m}\circ f,~ m\in\N,~ f\in \Aff\left(E\right)\right\}$. Since the composition is continuous, it follows that $\overline{\spa}\left\{\nu\circ\Phi\circ \mathbf{g},~ \nu\in F^{*},~ \mathbf{g}\in \Aff\left(E,E\right)\right\}$ contains $\left\{\varphi\circ f, ~ f\in \Aff\left(E\right)\right\}$, whose span is dense by Corollary \ref{Eugenue}. This contradiction shows that there is $m\in\N$ such that $\nu\circ\Phi\circ\mathbf{h}\in P_{m}\left(\R\right)$, for every $\nu\in F^{*}$ and $\mathbf{h}\in\Aff\left(\R,E\right)$, as required.\medskip

It is trivial that (ii) implies the last condition. We now prove that this condition implies (i) if $E$ and $F^{*}$ are Baire. Fix $\mathbf{h}\in\Aff\left(\R,E\right)$ and for $m\in\N$ let $H_{m}$ be the set of $\nu\in F^{*}$ such that $\nu\circ\Phi\circ\mathbf{h}\in P_{m}\left(\R\right)$. It is easy to see that $H_{m}$ is weak* closed subspace of $F^{*}$, and that $\bigcup\limits_{m\in\N}H_{m}=F^{*}$. Since $F^{*}$ is Baire, it follows that $H_{m}=F^{*}$, for some $m\in\N$, and so by \cite[Corollary 4]{bs} that $\Phi\circ\mathbf{h}$ is a polynomial. As $\mathbf{h}$ was arbitrary, it follows from \cite[Proposition 4]{bs} that $\Phi$ is a polynomial.
\end{proof}

Let $K$ be a compact Hausdorff space and $F$ be a Banach space. The space of continuous functions $\Co\left(K,F\right)$ from $K$ to $F$ is a Banach space under the pointwise algebraic operations and the norm $\|\mathbf{h}\|=\bigvee\limits_{x\in K}\|\mathbf{h}\left(x\right)\|$, for $\mathbf{h} \in\Co\left(K,F\right)$.\medskip

For $g \in\Co\left(K\right)$ and $f \in F$ we denote by $g\otimes f$ the element of $\Co\left(K,F\right)$ defined by  $\left[g\otimes f\right]\left(x\right):=g\left(x\right)f$ for all $x\in K$. Note that $\|g\otimes f\|=\|g\|\|f\|$. For $A \subset \Co\left(K\right)$ we define $A\otimes F:=\left\{g \otimes f, g \in A, f \in F\right\}$. Note that $\spa\left(\Co\left(K\right) \otimes F\right)$ is dense in $\Co\left(K,F\right)$ (see e.g. \cite[Theorem 6.1.2]{L03}).\medskip

If $X$ is a Tychonoff space, we endow $\Co\left(X,F\right)$ with the topology of uniform convergence on compact sets in a similar fashion as in the compact case. We now refine the just mentioned result as follows.

\begin{lemma}\label{dens}
Let $X$ be a Tychonoff space and let $F$ be a Banach space. If $A\subset\Co\left(X\right)$ is such that $\spa A$ is dense  in $\Co\left(X\right)$, then $\spa\left(A \otimes F\right)$ is dense in $\Co\left(X,F\right)$.
\end{lemma}
\begin{proof}
Let $\mathbf{h} \in\Co\left(X,F\right)$, let $K\subset X$ be compact, and let $\varepsilon>0$. By the result quoted above there
is $\mathbf{g}\in\spa\left(\Co\left(K\right) \otimes F\right)$ such that $\|\left.\mathbf{h}\right|_{K}-\mathbf{g}\|\le\varepsilon$. We have that $\mathbf{g}=\sum\limits_{k=1}^{n}g_{k}\otimes f_{k}$, where $g_{k}\in \Co\left(K\right)$, and $f_{k}\in F$, for every $k=1,...,n$. For $k=1,...,n$, by Tietze extension theorem (applied to $K$ in any Hausdorff compactification of $X$) there is $\hat{g}_{k}\in \Co\left(X\right)$ which extends $g_{k}$, and by assumption there is $e_{k}\in \spa A$ which is $\frac{\varepsilon}{n\|f_{k}\|}$-close to $\hat{g}_{k}$ on $K$. Let $\mathbf{e}:=\sum\limits_{k=1}^{n}e_{k}\otimes f_{k}\in \spa\left(\left(\spa A\right)\otimes F\right)=\spa\left(A \otimes F\right)$. It follows that $\mathbf{e}$ is $2\varepsilon$-close to $\mathbf{h}$ on $K$. As $K$ and $\varepsilon$ were arbitrary, the claim follows.
\end{proof}

Combining this with Corollary \ref{Eugenue} yields the following.

\begin{corollary}\label{foivos}
Let $E$ be a locally convex space, let $F$ be a Banach space, and let $\varphi\in\Co\left(\R\right)$ be non-polynomial. Then, $$\spa\left(\left\{\varphi\circ g,~ g\in \Aff\left(E\right)\right\}\otimes F\right)$$ is dense in $\Co\left(E,F\right)$.
\end{corollary}

\section{Applications to vector lattice theory}

In this section we present some instances when the results of the preceding section might be helpful in the context of vector lattice theory.

\subsection{Relatively uniformly closed sublattice generated by a set}\label{subhui}

Let $F$ be an Archimedean vector lattice. There are multiple formulas to describe the (linear) sublattice of $F$ generated by a given set $A\subset F$ (see e.g. \cite[Proposition 2.6]{dw}). The case when $A\subset F_{+}$ contains only two elements is particularly simple. Namely, it was proved in \cite[Theorem 2.1]{bhp} that if $u,v\in F_{+}$, then $\vlt\left\{u,v\right\}=\spa \left\{\left(ru+sv\right)^{+},~r,s\in\R\right\}=\spa\left(\spa\left\{u,v\right\}\right)^{+}$ (see also proofs in \cite[Theorem 1.1]{huijsmans} or \cite[Theorem 3.1]{dw}, and applications of this formula in e.g. \cite{gl}). It was observed in \cite[Examples 2.3 and 2.4]{huijsmans} that the equality $\vlt\left(A\right)=\spa\left(\spa\left(A\right)\right)^{+}$ may fail if $A$ contains two non-positive elements, or three positive elements. Versions of this formula for sets with sufficiently many disjoint vectors are established in \cite[Section 3]{dw}.

In this section we attempt to somewhat rescue the situation by taking the closure. For this we need to introduce the suitable notion of convergence.

For $e\in F_{+}$ the principal ideal $F_{e}$ is endowed with the norm $\|\cdot\|_{e}$ defined by $\|f\|_{e}=\bigwedge\left\{\alpha\ge 0:~ \left|f\right|\le\alpha e\right\}$. A net $\left(f_{\alpha}\right)_{\alpha\in A}\subset F$ converges \emph{uniformly} to $f$ \emph{relative} to $e\in F$ if $f\in F_{e}$ and for every $\varepsilon>0$ there is $\alpha_{0}$ such that $\|f-f_{\alpha}\|_{e}\le\varepsilon$ (and in particular $f_{\alpha}\in F_{e}$), for every $\alpha\ge \alpha_{0}$. We denote it by $f_{\alpha}\xrightarrow[]{\|\cdot\|_{e}}f$ and say that $\left(f_{\alpha}\right)_{\alpha\in A}\subset F$ converges \emph{relatively uniformly} to $f$, if $f_{\alpha}\xrightarrow[]{\|\cdot\|_{e}}f$, for some $e$.\medskip

Throughout the section we fix $\Phi:F\to F$ defined by $\Phi\left(f\right)=rf^{+}+sf^{-}$, for some $r,s\in\R$ with $s\ne -r$. Note that if $r=1$, and $s=0$ we get the positive part (a.k.a. ReLU), and if $r=s=1$ we get the absolute value. We do not allow $s=-r$, because then we would have $\Phi\left(f\right)=rf$, for every $f\in F$.

\begin{theorem}\label{spanning_lemma}
If $F$ is an Archimedean vector lattice and $A\subset F_{+}$ then $\spa\Phi\left(\spa A\right)$ is a relatively uniformly dense subspace of $\vlt\left(A\right)$.
\end{theorem}
\begin{proof}
It is clear that $\spa\Phi\left(\spa A\right)\subset \vlt\left(A\right)$. In order to prove its density it is enough to consider the case when $A$ is finite, since every element of $\vlt\left(A\right)$ is contained in $\vlt\left(A'\right)$, for some finite $A'\subset A$ (which depends on the given element).

Thus, assume that $A=\left\{a_{1},...,a_{n}\right\}$, and let $e:=a_{1}+...+a_{n}$. We have that $\vlt\left(A\right)\subset F_{e}$. It is enough to show that every element of $\vlt\left(A\right)$ can be approximated by elements of $\spa\varphi\left(\spa A\right)$ with respect to $\|\cdot\|_{e}$, and so without loss of generality we may assume that $F=F_{e}$. Hence, we may further assume that $F$ is a dense sublattice of $\Co\left(K\right)$, for some compact Hausdorff $K$ and $e=\1$ (see \cite[Theorem 45.3]{lz}), so that, in particular, $A\subset\left[\0,\1\right]$. Also, $\|\cdot\|_{e}$ is now simply the supremum norm on $K$. Finally, if $\varphi:\R\to\R$ is defined by $\varphi\left(t\right)=rt^{+}+st^{-}$, then for $f\in F\subset \Co\left(K\right)$ we have that $\Phi\left(f\right)=\varphi\circ f$.

Let $\mathbf{a}:K\to\left[0,1\right]^{n}$ be defined by $\mathbf{a}\left(x\right):=\left(a_{1}\left(x\right),...,a_{n}\left(x\right)\right)$. Note that $\vlt\left(A\right)=\left\{h\circ\mathbf{a},~h\in LL\left(\R^{n}\right)\right\}$, where $LL\left(\R^{n}\right)$ stands for the space of lattice-linear functions on $\R^{n}$.

We claim that $\varphi\circ\Aff\left(\R^{n}\right)\circ \mathbf{a} \subset \Phi\left(\spa A\right)$. Let $g \in \Aff\left(\R^n\right)$ so that $g\left(y\right)=u\cdot y+b$, $y\in\R^{n}$, where $u=\left(u_{1},\cdots,u_{n}\right) \in \R^{n}$ and $b \in \R$.
For every $x \in K$ we have that $$\left[\varphi \circ g \circ  \mathbf{a}\right]\left(x\right)=\varphi\left(\sum_{i=1}^n u_{i}a_{i}\left(x\right)+b\right)=\left[\Phi\left(f\right)\right]\left(x\right),$$ where $f=\sum_{i=1}^{n} u_{i}a_{i}+b\sum_{i=1}^{n} a_{i} \in \spa A$. Therefore $$\spa\left\{\varphi\circ \Aff\left(\R^n\right)\right\}\circ \mathbf{a} \subset \spa\left\{\varphi\circ\Aff\left(\R^n\right)\circ \mathbf{a}\right\}\subset \spa\left\{ \Phi\left(\spa A\right)\right\}.$$

Finally, fix $h\in LL\left(\R^{n}\right)$ and $\varepsilon>0$. By Corollary \ref{Eugenue} there is $g \in \spa\left\{\varphi\circ \Aff\left(\R^n\right)\right\}$ such that $\|h-g\|_{\left[0,1\right]^{n}}\le\varepsilon$. Then, $\|h\circ\mathbf{a}-g\circ\mathbf{a}\|\le\varepsilon$ and the result follows. 
\end{proof}

We call a set $A\subset F$ \emph{relatively uniformly closed} if it contains all existing relative uniform limits of the nets in $A$. The intersection of all relatively uniformly closed sets containing $B\subset F$ is its \emph{relative uniform closure} $\overline{B}$. We will need the fact that if $E\subset F$ is a sublattice, then $\overline{E}$ is also a sublattice (see \cite[Theorem 63.1(ii)]{lz}). Hence, the smallest relatively uniformly closed sublattice containing $B\subset F$ is $\overline{\vlt\left(B\right)}$. Applying Theorem \ref{spanning_lemma} yields the following.

\begin{corollary}
If $F$ is an Archimedean vector lattice and $A\subset F_{+}$ then\linebreak $\overline{\vlt\left(A\right)}=\overline{\spa\Phi\left(\spa A\right)}$. In particular, $\overline{\spa\Phi\left(\spa A\right)}$ is the relatively uniformly closed sublattice generated by $A$ in $F$.
\end{corollary}

In a Banach lattice a set is norm closed iff it is relatively uniformly closed (it follows from \cite[Theorem 15.6]{zaanen}).

\begin{corollary}\label{formula}
If $F$ is a Banach lattice and $A\subset F_{+}$ then $\overline{\spa}\Phi\left(\spa A\right)$ is the closed sublattice generated by $A$ in $F$.
\end{corollary}

Let us apply the preceding result to the Banach lattice $\Co_{ph}\left(\R^{n}\right)$ of positively homogeneous continuous functions on $\R^{n}$ endowed with the supremum-norm on $\left[-1,1\right]^{n}$. Note that the topology that this norm generates agrees with the restriction of the topology on $\Co\left(\R^{n}\right)$ to $\Co_{ph}\left(\R^{n}\right)$. Let $e_{1},...,e_{n}$ be the coordinate projections on $\R^{n}$.


\begin{corollary}\label{cybph}
$\spa\Phi\left(\spa\left\{e_{1}^{\pm},...,e_{n}^{\pm}\right\}\right)$ is norm-dense in $\Co_{ph}\left(\R^{n}\right)$.
\end{corollary}

Note that in this context $\Phi\left(f\right)=\varphi\circ f$, where $\varphi:\R\to\R$ is defined by $\varphi\left(t\right)=rt^{+}+st^{-}$.

\begin{question}
What is $\overline{\spa}\left\{\varphi\circ f,~ f\in \R^{n*}\right\}$?
\end{question}

\begin{remark}
Let $E$ be a directed Archimedean ordered vector space considered as a subspace of its Dedekind completion $E^{\delta}$. Then, $E=E_{+}-E_{+}$, and so $E^{\rho\mathrm{ru}}:=\overline{\spa}\Phi\left(E\right)$ is the relative uniform completion (see \cite{bt}) of the Riesz completion of $E$ (see \cite[Section 2.4]{kv}). One can show that this is the unique relatively uniformly complete vector lattice which satisfies the following universal property: if $F$ is a relatively uniformly complete Archimedean vector lattice, and $T:E\to F$ is a Riesz* homomorphism (see \cite[Section 2.3]{kv}), there is a unique vector lattice homomorphism $T^{\rho\mathrm{ru}}:E^{\rho\mathrm{ru}}\to F$ which extends $F$ (combine \cite[Theorem 2.4.11]{kv} with \cite[Theorem 5.3]{bt}).
\qed\end{remark}

We conclude the section with a slightly more detailed exposition and a minor generalization of \cite[Example 2.4]{huijsmans}.

\begin{example}
Let $F=\Co\left(\R^{3}_{+}\right)$ and let $u,v,w:\R^{3}_{+}\to \R_{+}$ be the coordinate projections. Then, $\spa\left\{u,v,w\right\}$ consists of the restrictions of linear functions. Clearly, $g:=\left(u\wedge v - w\right)^{+}\in\vlt\left\{u,v,w\right\}$. Let us show that $g\notin \spa\Phi\left(\spa \left\{u,v,w\right\}\right)$\linebreak $=\spa\left(\varphi\circ\spa \left\{u,v,w\right\}\right)$, where $\varphi\left(t\right)=rt^{+}+st^{-}$. Assume the contrary. Then, grouping together terms which come from proportional linear functions one can show that $g=\sum\left(r_{i}e_{i}^{+}-s_{i}e_{i}^{-}\right)$, where $e_{i}$'s are restrictions of non-proportional linear functions and $r_{i},s_{i}\in\R$. Expanding further yields $$g=\sum\left(\left(r_{i}-s_{i}\right)e_{i}^{+}+s_{i}\left(e_{i}^{+}-e_{i}^{-}\right)\right)=\sum\left(\left(r_{i}-s_{i}\right)e_{i}^{+}+s_{i}e_{i}\right)=h+\sum\pm f_{i}^{+},$$ where $h$ is linear, and $f_{i}$'s are linear and mutually non-proportional. It follows from this representation that $g$ is not differentiable at each point where exactly one of $f_{i}$'s vanishes, and differentiable outside of $\bigcup\ker f_{i}$.

Let $L:=\left\{\left(1-2t,t,t\right),~ 0<t<\frac{1}{3}\right\}$ and let $M:=\left\{\left(1-2t,t,t\right),~ \frac{1}{3}<t<\frac{1}{2}\right\}$, which are parts of a line segment in $\R^{3}_{+}$. Note that $g$ is not differentiable at every point of $L$. Indeed, for $0<s<1-3t$ we have  $\frac{1}{s}\left(g\left(1-2t,t+s,t\right)-g\left(1-2t,t,t\right)\right)=\frac{1}{s}\left(s-0\right)=1$, while for $-t<s<0$ we have  $\frac{1}{s}\left(g\left(1-2t,t+s,t\right)-g\left(1-2t,t,t\right)\right)=\frac{1}{s}\left(0-0\right)=0$. Hence, $L\subset \bigcup\ker f_{i}$. In fact, if two points of $L$ are contained in the plane $\ker f_{i}$, for some $i$, then the entire segment is contained in $\ker f_{i}$, including $M$.

On the other hand, $g$ vanishes on a neighborhood of every point of $M$, and so it is differentiable there. This means that $M$ has to be covered by $\bigcup\limits_{j\ne i}\ker f_{j}$. However, since $f_{i}$'s are non-proportional, it follows that the pairwise intersections of their kernels are contained in rays emanating from the origin, and so cannot collectively cover $M$. Contradiction.
\qed\end{example}

\subsection{Other applications}\label{suboth}

In this subsection we gather some ``side applications'' of the Universal Approximation Theorem. In Section 2 of \cite{AT:17} the authors proved that the space of piecewise affine functions $\vlt\left(\Aff\left(\mathbb{R}^{m}\right)\right)$ is $\sigma$-order dense in $\Co\left(\R^{m}\right)$. Below we strengthen that result and present a vector valued extension.

For a Tychonoff space $X$ we denote the collection of all continuous functions on $X$ which vanish outside of a compact set by $\Co_{00}\left(X\right)$. Recall that $X$ is called \emph{hemi-compact} if there is a sequence $\left\{K_{n}\right\}_{n\in\N}$ of compact sets in $X$ such that every compact $K\subset X$ is contained in $K_{n}$, for some $n\in\N$. If $X$ is hemi-compact, and $Y$ is a metric space, then $\Co\left(X,Y\right)$ equipped with the compact open topology is metrizable (see \cite[Exercise 4.2.H]{engelking}). Every locally compact $\sigma$-compact space is hemi-compact (see \cite[Exercise 3.8.C]{engelking}), and if $Y$ is a complete metric space, then $\Co\left(X,Y\right)$ is completely metrizable (see \cite[Exercise 4.3.F]{engelking}).

\begin{proposition}\label{c00}
If $X$ is locally compact and $\sigma$-compact and $E$ is a dense sublattice of $\Co_{00}\left(X\right)$, then for every $f\in \Co\left(X\right)_{+}$ there is an increasing sequence in $E\cap\left[\0,f\right]$ which is $\mathrm{ru}$-convergent to $f$.
\end{proposition}
\begin{proof}
First, assume that $f\in \Co_{00}\left(X\right)$ so that $\overline{\supp} f$ is compact. Since $X$ is locally compact there is an open $U$ which contains $\overline{\supp} f$ and such that $K:=\overline{U}$ is compact.

Fix $\varepsilon>0$. There is $e\in E$ such that $\|f-e\|_{K}\le\varepsilon$, so that $\left.f\right|_{K}-\varepsilon\left.\1\right|_{K}\le \left.e\right|_{K}\le \left.f\right|_{K}+\varepsilon\left.\1\right|_{K}$. There is $h\in E$ such that $\|h-2\varepsilon\1\|_{K}\le\varepsilon$, so that $\varepsilon\left.\1\right|_{K}\le \left.h\right|_{K}\le 3\varepsilon\left.\1\right|_{K}$. It follows that $\left.f\right|_{K}-4\varepsilon\left.\1\right|_{K}\le \left.e-h\right|_{K}\le \left.f\right|_{K}$. Since $\overline{\supp} f$ and $\overline{\supp}\left(e-h\right)^{+}\backslash U$ are two disjoint compact sets, by \cite[Corollary 2.2]{erz2} there is $g\in E$ which agrees with $\left(e-h\right)^{+}$ on the former and vanish on the latter. Thus, $u:=\left(e-h\right)^{+}\wedge g^{+}$ agrees with $\left(e-h\right)^{+}$ on $\supp f$ and vanishes on $X\backslash \supp \left(e-h\right)^{+}$ as well as on $\overline{\supp}\left(e-h\right)^{+}\backslash U$, which implies $\supp u\subset U$. Since $\left(e-h\right)^{+}$ is dominated by $f$ on $K$, it follows that $u\in\left[\0,f\right]$ and $\left\|f-u\right\|_{\8}=\left\|f-u\right\|_{K}\le 4\varepsilon$.

It is easy to see that $\Co_{00}\left(X\right)$ is dense in $\Co\left(X\right)$. Since the topology is metrizable, there is a sequence $f_{n}\in \Co_{00}\left(X\right)$, which converges to $f$. Replacing $f_{n}$ with $f\wedge f_{n}^{+}$ if needed, we may assume that $f_{n}\in\left[\0,f\right]$, for every $n\in\N$. By the previous step for every $n\in\N$ there is $e_{n}\in E\cap \left[\0, f_{n}\right]\subset \left[\0,f\right]$ such that $\|e_{n}-f_{n}\|_{\8}\le \frac{1}{n}$. Then, $e_{n}\to f$, and so by \cite[Theorem 15.6]{zaanen} there is subsequence of $\left(e_{n}\right)_{n\in\N}$ which relatively uniformly converges to $f$. Taking suprema of the first members of this subsequence makes it increasing.
\end{proof}

\begin{corollary}\label{vladimir}
For every $f\in \Co\left(\R^{m}\right)_{+}$ there is an increasing sequence in $\vlt\left(\Aff\left(\mathbb{R}^{m}\right)\right)\cap\Co_{00}\left(\R^{m}\right)\cap\left[\0,f\right]$ which is $\mathrm{ru}$-convergent to $f$.
\end{corollary}
\begin{proof}
Let us show that $E:=\vlt\left(\Aff\left(\mathbb{R}^{m}\right)\right)\cap\Co_{00}\left(\R^{m}\right)$ strongly separates points of $\R^{m}$ in the sense that if $x,y\in\R^{m}$ are distinct, then there is $e\in E$ such that $e\left(x\right)=1$ and $e\left(y\right)=0$. Indeed, without loss of generality we may assume that $x=0_{\R^{m}}$ and the $\ell_{\8}$ norm of $y$ is greater than $1$. Let $e_{1},...,e_{m}$ be the coordinate projections on $\R^{m}$. Then, $e:=\left(\1-\left|e_{1}\right|\vee...\vee \left|e_{m}\right|\right)^{+}$ satisfies the requirements. According to \cite[Theorem 2.1]{BT:24} we have that $E$ is dense in $\Co\left(\R^{m}\right)$, and so the claim follows from Proposition \ref{c00}.
\end{proof}

\begin{theorem}\label{vladimir2}
Let $E$ be a locally convex space, let $F$ be a Banach lattice, and let $\Phi:F\to F$ be defined by $\Phi\left(f\right)=rf^{+}+sf^{-}$, for some $r,s\in\R$ with $s\ne -r$. Then, $\spa\left\{\Phi\circ \mathbf{g},~ \mathbf{g}\in \Aff\left(E,F\right)\right\}$ is dense in $\Co\left(E,F\right)$. In particular, $\vlt\left(\Aff\left(E,F\right)\right)$ is dense in $\Co\left(E,F\right)$.
\end{theorem}
\begin{proof}
In light of Corollary \ref{foivos} it suffices to show that $\varphi\circ g\otimes f\in \spa\left\{\Phi\circ \mathbf{g},~ \mathbf{g}\in \Aff\left(E,F\right)\right\}$, for every $g\in \Aff\left(E\right)$, and $f\in F$, where $\varphi\left(t\right)=rt^{+}+st^{-}$. Indeed, evaluating at each $e\in E$ and some simple computations yield $$\varphi\circ g\otimes f=\varphi\circ g\otimes f^{+}-\varphi\circ g\otimes f^{-}=\Phi\circ\left( g\otimes f^{+}\right)-\Phi\circ\left( g\otimes f^{-}\right);$$ as $g\otimes f^{\pm}\in \Aff\left(E,F\right)$, the result follows.
\end{proof}

\begin{remark}
The preceding theorem may be viewed as a vector valued version of Corollary \ref{Eugenue}. We remark that the reason we only consider $\Phi$ of a very special form is that we are interpreting functions from $\R$ to $\R$ as self-maps of $F$ through the functional calculus. For a general Banach lattice the only functional calculus available is by continuous positively homogeneous functions. A continuous positively homogeneous function on $\R$ which is not a polynomial is always of the form $\varphi\left(t\right)=rt^{+}+st^{-}$, for some $r,s\in\R$ with $s\ne -r$. It seems plausible that the density result akin to the one above holds for a wider class of maps $\Phi$, but some other methods are needed to attack this problem.
\qed\end{remark}

The second ``side application'' of the Universal Approximation Theorem is an alternative proof of \cite[Theorem 1.1]{sw} and its minor generalization (see also \cite{bhs} and the references therein for similar results).

\begin{corollary}\label{sw}
Let $X$ be a compact Hausdorff space, and let $F$ be a linear subspace of $\Co\left(X\right)$ which contains $\1$. Let $U\subset\R$ be an open interval, and let $\varphi\in\Co\left(U\right)\backslash \Aff\left(U\right)$ be such that $\varphi\circ f\in F$, for every $f\in F$ with $f\left(X\right)\subset U$. Then, $\overline{F}$ is a subalgebra and a sublattice of $\Co\left(X\right)$. It is dense whenever $F$ separates points of $X$.
\end{corollary}
\begin{proof}
Using continuity of the composition, it follows that $\varphi\circ f\in \overline{F}$, for every $f\in \overline{F}$ with $f\left(X\right)\subset U$. Hence, without loss of generality we may assume that $F$ is closed. Let $K:=\left[-1,1\right]\subset\R$.\medskip

\textbf{Case 1}. Assume that $\varphi$ is not a polynomial. Let $G_{\varphi}$ be as in the proof of Theorem \ref{lybenko}; by the theorem it is dense in $\Co\left(\R\right)$. If $g\in \Aff\left(\R\right)$ is such that $g\left(K\right)\subset U$ and $f\in \Bbo_{F}$, then $g\circ f\in F$, and $\left[g\circ f\right]\left(X\right)\subset U$ so that $\varphi\circ g\circ f\in F$. Hence, if $h\in G_{\varphi}$ and $f\in \Bbo_{F}$ then $h\circ f\in F$. Using continuity of the composition, it follows that if $h\in \overline{G_{\varphi}}$ and $f\in \Bbo_{F}$, then $h\circ f\in F$. In particular, since the square function is in $\overline{G_{\varphi}}$, if $e,f\in \frac{1}{2}\Bbo_{F}$, then $4ef=\left(e+f\right)^{2}-\left(e-f\right)^{2}\in F$. Finally, for general $e,f\in F$, we have $ef=\|e\|\|f\|\left(4\frac{1}{2\|e\|}e\frac{1}{2\|f\|}f\right)\in F$, and so $F$ is a subalgebra.\medskip

\textbf{Case 2}. Assume that $\varphi$ is a polynomial and represent $U$ as $\left(u,v\right)$, where $u<v$. Let $\varphi_{n}$ be defined by $\varphi_{n}\left(t\right):=n\left(\varphi\left(t+\frac{1}{n}\right)-\varphi\left(t\right)\right)$. We have that $\varphi_{n}\circ f\in F$, for every $f\in F$, with $f\left(X\right)\subset \left(u,v-\frac{1}{n}\right)$, and also $\varphi_{n}\to\varphi'$. Using continuity of composition again, we conclude that $\varphi'\circ f\in F$, for every $f\in F$ with $f\left(X\right)\subset \left(u,v\right)$ (note that since $f\left(X\right)$ is compact, $f\left(X\right)\subset \left(u,v-\frac{1}{n}\right)$ holds automatically for large enough $n$). Iterating the argument we get that $\varphi^{\left(m\right)}\circ f\in F$, for every $f\in F$ with $f\left(X\right)\subset U$ and $m\in\N$. Take $m:=\deg\varphi-2$; then $\varphi^{\left(m\right)}\left(t\right)=at^{2}+bt+c$, where $a\ne 0$ and $b,c\in\R$. It follows that if $f\in F$ satisfies $f\left(X\right)\subset U$, then $f^{2}=\frac{1}{a}\left(\varphi^{\left(m\right)}\circ f-bf-c\1\right)\in F$. For the general $f$, we have that $g:=\frac{v-u}{2\|f\|}f+\frac{v+u}{2}\1\in F$ satisfies $g\left(X\right)\subset U$, hence $g^{2}\in F$, thus $f^{2}\in F$. By the argument above we conclude that $F$ is a subalgebra.\medskip

Now every closed subalgebra of $\Co\left(X\right)$ is a sublattice (see \cite[Theorem 16.5.2]{bn}). The last claim follows from the classical Stone-Weierstrass theorem (see \cite[Theorem 16.5.7(2)]{bn}).
\end{proof}

For example, if $F\ni\1$ is such that $f^{-1}\in F$, for any non-vanishing $f\in F$, then $\overline{F}$ is a subalgebra and a sublattice of $\Co\left(X\right)$. The following version of the preceding result is proven similarly, but using Theorem \ref{cybenko} instead of Theorem \ref{lybenko}.

\begin{corollary}
Let $X$ be a Tychonoff space, and let $F$ be a linear subspace of $\Co\left(X\right)$ which contains $\1$. Let $\varphi\in\Co\left(\R\right)\backslash \Aff\left(\R\right)$ be such that $\varphi\circ f\in F$, for every $f\in F$. Then, $\overline{F}$ is a subalgebra and a sublattice of $\Co\left(X\right)$. It is dense whenever $F$ separates points of $X$.
\end{corollary}

\subsection{Other approximation results}\label{submisc}

In this section we add some miscellaneous results, which are not directly related to the Universal Approximation Theorems, but nevertheless are similar in spirit to the material from the preceding subsection.

We start with a generalization of \cite[Lemma 3.4]{DMAZ:21}. For a vector lattice $E$ we denote the collection of all $\R$-valued vector lattice homomorphisms on $E$ by $\Hom\left(E\right)$. It is not true in general that if $X$ is a Tychonoff space, then $\Hom\left(\Co\left(X\right)\right)\subset \Co\left(X\right)^{*}$. We always have that $\Hom\left(\Co\left(X\right)^{*}\right)\cap \Co\left(X\right)^{*}$ consists of positive scalar multiples of point evaluations. Indeed, this follows from the fact that closed ideals in $\Co\left(X\right)$ of co-dimension $1$ are of the form $\left\{f\in \Co\left(X\right),~ f\left(x_{0}\right)=0\right\}$, for some $x_{0}\in X$ (see e.g. \cite[Corollary 2.3]{BT:24}).

\begin{proposition}\label{lem}
Let $E$ be a Banach lattice, let $X$ be a Tychonoff space, and let $\nu$ be a non-zero continuous lattice homomorphism on $\Co\left(X,E\right)$. Then, there are unique $x\in X$ and $\mu\in\Hom\left(E\right)$ such that $\nu\left(\mathbf{f}\right)=\mu\left(\mathbf{f}\left(x\right)\right)$, for every $\mathbf{f}\in \Co\left(X,E\right)$.
\end{proposition}
\begin{proof}
It is easy to see that the correspondence $E\ni e\mapsto \nu\left(\1\otimes e\right)$ is a lattice homomorphism on $E$; call it $\mu$. For every $e\in E_{+}$ the correspondence $\Co\left(X\right)\ni f\mapsto \nu\left(f\otimes e\right)$ is a continuous lattice homomorphism on $\Co\left(X\right)$, and so there are $r_{e}\ge 0$ and $x_{e}\in X$ such that $\nu\left(f\otimes e\right)=r_{e}f\left(x_{e}\right)$, for every $f\in\Co\left(X\right)$. Moreover, if $r_{e}\ne 0$, then $x_{e}$ is uniquely determined. It also follows that $\mu\left(e\right)=\nu\left(\1\otimes e\right)=r_{e}\1\left(x_{e}\right)=r_{e}$, for every $e\in E_{+}$.

Assume that $e,h\in E_{+}$ are such that $x_{e}\ne x_{h}$. Find disjoint $f,g\in\Co\left(X\right)$ such that $f\left(x_{e}\right)=1=g\left(x_{h}\right)$. Then, $f\otimes e\bot g\otimes h$, hence one of $\nu\left(f\otimes e\right)=\mu\left(e\right)f\left(x_{e}\right)=\mu\left(e\right)$ and $\nu\left(g\otimes h\right)=\mu\left(h\right)$ is equal to $0$. Therefore, if $e,h\in E_{+}\backslash\ker\mu$ then $x_{e}=x_{h}$. Thus, there is a unique $x\in X$ such that $x_{e}=x$, for every $e\in E_{+}\backslash\ker\mu$.

We now have that $\nu\left(f\otimes e\right)=\mu\left(e\right)f\left(x\right)$, for every $f\in\Co\left(X\right)$ and $e\in E_{+}\backslash\ker\mu$. On the other hand, if $e\in \ker \mu$, then $\nu\left(f\otimes e\right)=\mu\left(e\right)f\left(x_{e}\right)=0=\mu\left(e\right)f\left(x\right)$. It follows that $\nu$ agrees with $\mu\circ\delta_{x}$ on $\spa\left(\Co\left(X\right)\otimes E_{+}\right)$, which is dense in $\Co\left(X,E\right)$, according to Lemma \ref{dens}. Since both functionals are continuous, we conclude that they are equal. Uniqueness follows from the construction.
\end{proof}

Let $F$ be a vector lattice. We will say that $A\subset F$ is an \emph{AM-set} if it is solid and $f,g\in A_{+}$ $\Rightarrow$ $f\vee g\in A$. Note that every AM-set is convex: if $f,g\in A$ and $0\le t\le 1$, then $\left|tf+\left(1-t\right)g\right|\le \left|f\right|+\left(1-t\right)\left|g\right|\le \left|f\right|\vee \left|g\right|\in A$.

We will call a linear topology on $F$ an \emph{AM-topology} if it has a local base at $0_{F}$ consisting of AM-sets. Clearly, such a topology is locally convex. Taking Minkowski functionals of the AM-sets in the base yields a collection of AM-seminorms which generates the topology (we will say that a solid seminorm $\rho$ on $F$ is an \emph{AM-seminorm} if $\rho\left(f\vee g\right)=\rho\left(f\right)\vee \rho\left(g\right)$, for all $f,g\in F_{+}$). Conversely, any topology generated on $F$ by a collection of AM-seminorms is an AM-topology.

In the next result we will use Arzela-Ascoli theorem (see \cite[Theorem 0.4.11]{edwards} and \cite[Theorem 8.2.10]{engelking}), which partially requires the underlying topological space to be \emph{compactly generated} (see \cite[Section 3.3]{engelking}). We will not go into details regarding this class of spaces, it suffices to mention that every sequential (including metrizable) Tychonoff space is compactly generated.

\begin{proposition}\label{am}
Let $F$ be a vector lattice endowed with a Hausdorff AM-topology, which is also compactly generated. Then, $F$ embeds algebraically and topologically as a sublattice of $\Co\left(X\right)$, for a Tychonoff space $X$.
\end{proposition}
\begin{proof}
Let $X_{F}$ be the collection of all continuous lattice homomorphisms on $F$ endowed with the compact-open topology (so we consider $X_{F}$ to be a subset of $\Co\left(F\right)$; it is easy to see that it is closed). According to Arzela-Ascoli theorem if $A\subset \Co\left(F\right)$ is equicontinuous and pointwise bounded, then it is relatively compact. Note that $A\subset F^{*}$ is equicontinuous if and only if $A^{0}:=\left\{f \in F,~ \left|a\left(f\right)\right| \le 1,~ \text{for all}~ a \in A\right\}$  is a neighborhood of $0_{F}$, which also implies pointwise boundedness. Hence, if $K\subset X_{F}$ is such that $K^{0}$ is a neighborhood of $0_{F}$, then it is relatively compact. On the other hand, if $F$ is compactly generated, the converse to Arzela-Ascoli theorem also holds, and so in this case $K\subset X_{F}$ is relatively compact if and only if $K^{0}$ is a neighborhood of $0_{F}$.

Define $J:F\to  \Co\left(X_{F}\right)$ by $\left[Jf\right]\left(x\right):=x\left(f\right)$. It is easy to see that $J$ is a vector lattice homomorphism. For every compact $K\subset X_{F}$, since $K^{0}$ is a neighborhood of $0_{F}$, there is a semi-norm $\rho$ such that $\Bo_{\rho}\subset K^{0}$, where $\Bo_{\rho}$ is the open unit ball with respect to $\rho$. The last condition is equivalent to the fact that $\|Jf\|_{K}\le \rho\left(f\right)$, for every $f\in F$. Since $K$ was chosen arbitrarily, it follows that $J$ is continuous.

On the other hand, let $\rho$ be an AM-seminorm of $F$ and let $H:=\ker\rho$, which is a closed ideal. Let $Q$ be the quotient map from $F$ onto $F\slash H$. Let us show that $\rho$ induces an AM-norm on $F\slash H$ by showing that $Q\Bo_{\rho}$ is an AM-set. First, a surjective homomorphic image of a solid set is solid. Next, if $u,v\in \left(Q\Bo_{\rho}\right)_{+}=Q\Bo_{\rho}^{+}$, there are $g,h\in \Bo_{\rho}^{+}$ such that $u=Qg$ and $v=Qh$. Then, $g\vee h\in \Bo_{\rho}$, and so $u\vee v=Q\left(g\vee h\right)\in Q\Bo_{\rho}$.

Let $\|\cdot\|$ be the norm on $F\slash H$ induced by $\rho$. Note that the quotient topology on $F\slash H$ is weaker than the norm-topology, and so $Q$ is a continuous homomorphism from $F$ onto $\left(F\slash H,\|\cdot\|\right)$. The completion $F_{\rho}$ of $\left(F\slash H,\|\cdot\|\right)$ is an AM-space, and so by  \cite[Theorem 4.29]{AL:06} there exists a compact Hausdorff space $L$ and a lattice isometry $T:E \to \Co\left(L\right)$. It can be easily deduced from here that $K_{\rho}:=\Hom\left(F_{\rho}\right)\cap\Bbo_{F_{\rho}^{*}}$ is norming on $F_{\rho}$, in other words, $\|u\|=\bigvee\limits_{\nu\in K_{\rho}}\left|\nu\left(u\right)\right|$. Also note that $K_{\rho}$ is equicontinuous on $F_{\rho}$. It follows that $Q^{*}K_{\rho}$ is an equicontinuous subset of $X_{F}$ such that $\left\|Jf\right\|_{Q^{*}K_{\rho}}=\|Qf\|=\rho\left(f\right)$, for every $f\in F$. Since $\rho$ was chosen arbitrarily, and AM-seminorms generate the topology on $F$ this shows that $J$ is a topological embedding.
\end{proof}

In order to get rid of the additional assumption that the topology is compactly generated, we need to modify the only part of the argument where this assumption was used.

\begin{question}
Let $F$ be a vector lattice endowed with a Hausdorff AM-topology. Is it true that if $K_{F}\subset\Hom\left(F\right)\cap F^{*}$ is compact in $\Co\left(F\right)$, then it is equicontinuous on $F$?
\end{question}

We can now obtain a version of \cite[Theorem 2.1]{BT:24} for vector lattices with AM-topologies. Note that in the following result $F^{\bot}$ is meant inside of the continuous dual $E^{*}$.

\begin{corollary}\label{thm0}
Let $E$ be a vector lattice endowed with a Hausdorff compactly generated AM-topology, and let $F\subset E$ be a sublattice. Then, $\overline{F}=\left(F^{\bot}\cap \left(\Hom\left(E\right)-\Hom\left(E\right)\right)\right)_{\bot}$. In particular $F$ is dense if and only if it separates the points of $\Hom\left(E\right)\cap E^{*}$.
\end{corollary}
\begin{proof}
Let $H:=\left(F^{\bot}\cap \left(\Hom\left(E\right)-\Hom\left(E\right)\right)\right)_{\bot}$. It is clear that $\overline{F}=F^{\bot}_{~\bot}\subset H$. To prove the converse by Proposition \ref{am} we may assume that $E$ is a sublattice of $\Co\left(X\right)$, for a Tychonoff space $X$. Since continuous homomorphisms on $\Co\left(X\right)$ restrict to continuous homomorphisms on $E$, it follows that $H\subset \left(F^{\bot}\cap \left(\Hom\left(\Co\left(X\right)\right)-\Hom\left(\Co\left(X\right)\right)\right)\right)_{\bot}=\overline{F}$, where the last equality follows from \cite[Theorem 2.1]{BT:24}.
\end{proof}

\begin{corollary}\label{prop1}
Let $X$ be a hemi-compact space and let $E$ be an AM-space. Assume that $F$ is a sublattice of $\Co\left(X,E\right)$ which contains all the $E$-valued constant functions and $\left\{\mu \circ \mathbf{f}, \mathbf{f} \in F\right\}$ separates the points of $X$ for all $\mu \in \Hom\left(E\right)\backslash\left\{0_{E^{*}}\right\}$. Then, $F$ is dense in $\Co\left(X,E\right)$.
\end{corollary}

\begin{proof}
It is easy to see that the compact open topology on $\Co\left(X,E\right)$ is a metrizable AM-topology, thus in view of Corollary \ref{thm0} for topological density it suffices to show that $F$ separates the continuous lattice homomorphisms on $\Co\left(X,E\right)$. Let $\nu_{1},\nu_{2}$ be distinct homomorphisms of $\Co\left(X,E\right)$. According to Proposition \ref{lem} there exist $x_{1},x_{2}\in X$ and $\mu_{1},\mu_{2}\in \Hom\left(E\right)$ such that $\nu_{i}\left(\mathbf{f}\right)=\mu_{i}\left(\mathbf{f}\left(x_{i}\right)\right)$ for all $\mathbf{f} \in \Co\left(X,E\right)$ and each $i=1,2$. Suppose that $\mu_{1} \ne \mu_{2}$. Pick $e \in E$ such that $\mu_{1}\left(e\right) \ne \mu_{2}\left(e\right)$ and note that $\1\otimes e\in F$. We have $$\nu_{1}\left(\1\otimes e\right)=\mu_{1}\left(e\right)\ne \mu_{2}\left(e\right)=\nu_{2}\left(\1\otimes e\right).$$

Suppose that $\mu_{1} = \mu_{2}=:\mu$, which implies $\mu\ne 0_{E^{*}}$ and $x_{1}\ne x_{2}$. By our assumption there is  $\mathbf{f} \in F$ such $$\nu_{1}\left(\mathbf{f}\right)=\mu\left(\mathbf{f}\left(x_{1}\right)\right) \ne \mu\left(\mathbf{f}\left(x_{2}\right)\right)=\nu_{2}\left(\mathbf{f}\right).$$ Thus $F$ separates the continuous homomorphisms and the result follows from Corollary \ref{thm0}. Relative uniform density now follows from \cite[Theorem 15.6]{zaanen}.
\end{proof}

\begin{remark}
The assumption that $E$ is an AM-space in Corollary \ref{prop1} is not superfluous. If $E$ is atomless order continuous Banach lattice, then $\Hom\left(E\right)=\left\{0_{E^{*}}\right\}$ (see e.g. \cite[Corollary 5.22]{BT:24}), and so the second assumption about $F$ in the corollary is trivial. Hence, the set $F=\left\{\1\otimes e,~ e\in E\right\}$ satisfies both of the assumptions, but is not dense aside of the case when $X$ is a singleton.
\qed\end{remark}

\begin{remark}\label{timur}
It follows from Corollary \ref{thm0} that for AM-spaces the sets in \cite[Theorem 5.9]{BT:24} are equal to the norm closure (which is the same as relative uniform closure) of a sublattice. Another class of spaces for which this happens is that of atomic order continuous Banach lattices (see \cite[Section 3]{BT:24}). We will now construct a closed proper ideal $F$ of a Banach lattice $E$ such that $\Hom\left(E\right)$ separates the points of $E$, but $F^{\bot}\cap \left(\Hom\left(E\right)-\Hom\left(E\right)\right)=\left\{0_{E^{*}}\right\}$, thus providing a counterexample to \cite[Question 5.10]{BT:24}.

Let $H$ be a Banach lattice with $\Hom\left(H\right)=\left\{0_{H^{*}}\right\}$ (see e.g. the preceding remark). Let $E:=\mathrm{FBL}\left[H\right]$ (see \cite{art}). Note that $E$ is contained in the space of positively homogeneous functions on $H^{*}$, and contains $H$ viewed as a space of functions on $H^{*}$. Moreover, according to \cite[Corollary 2.7]{art} lattice homomorphisms on $E$ are precisely evaluations at the points of $H^{*}$. Let $F\subset E$ be as in \cite[Proposition 3.8]{dhp}, and the discussion after it; this is a closed proper ideal of $E$ such that no element of $\Hom\left(E\right)$ vanishes on $F$. Assume that $\nu,\mu\in H^{*}$ are distinct. Our goal is to show that there is $f\in F$ such that $f\left(\nu\right)\ne f\left(\mu\right)$.

As $\nu\ne\mu$ we may assume that $\nu\ne 0_{H^{*}}$. Since no point evaluation vanishes on $F$, there is $f\in F_{+}$ such that $f\left(\nu\right)=1$. If $\mu=r\nu$, where $1\ne r\ge 0$, then $f\left(\mu\right)=re\left(\mu\right)=r\ne 1=f\left(\nu\right)$. If $\mu$ is not a positive multiple of $\nu$ there is $h\in H$ such that $\nu\left(h\right)=1$ and $\mu\left(h\right)\le 0$. Let $g:=\left(f-h^{+}\right)^{+}\in F$; we have that $f\left(\mu\right)=g\left(\mu\right)$ and $f\left(\nu\right)=1\ne 0=g\left(\nu\right)$. Hence, either $f\left(\nu\right)\ne f\left(\mu\right)$, or $g\left(\nu\right)\ne g\left(\mu\right)$.
\qed\end{remark}

\subsection*{Acknowledgements}

The authors would like to thank Vladimir Troitsky for valuable conversations on the topic of the paper, Karim Boulabiar for bringing Corollary \ref{sw} to the authors' attention and Timur Oikhberg for assistance with Remark \ref{timur}, and in particular bringing \cite[Proposition 3.8]{dhp} to the authors' attention. The first author was  supported by Pacific Institute for the Mathematical Sciences. The second author was supported by an NSERC discovery grant.

\end{document}